\documentclass[10pt]{amsart}
\usepackage{amssymb,latexsym,amsmath,amsthm,enumitem,mathrsfs,geometry}
\usepackage{fullpage}
\usepackage{fancyhdr}
\usepackage{comment}
\usepackage{mathrsfs}
\usepackage{hyperref}
\usepackage{lineno}
\usepackage{diagbox}
\usepackage{array}
\usepackage{tikz}
\usetikzlibrary{arrows}
\usetikzlibrary{positioning}
\usepackage{float}

\newtheorem{thm}{Theorem}[section]
\newtheorem*{thm*}{Theorem}

\newtheorem{lemma}[thm]{Lemma}

\renewcommand{\b}[1]{\mathbf{#1}}

\newcommand{\beq}{\begin{equation}}
\newcommand{\eeq}{\end{equation}}
\newtheorem{remark}{Remark}
\usepackage[titletoc,title]{appendix}
\def\A{\mathbb{A}}
\newcommand{\Q}{\mathbb{Q}}
\newcommand{\R}{\mathbb{R}}

\newcommand{\C}{\mathbb{C}}

\newcommand{\Nr}{\mathrm{N}}

\def\o{\operatorname{o}}
\def\i{\tilde{i}}
\def\b{\mathtt b}

\def\gp{\mathfrak p}

\newcommand{\GL}{\mathrm{GL}}

\newcommand{\kommentar}[1]{}
\newtheorem{theorem}{Theorem}
\newtheorem{example}{Example}

\newtheorem*{remark*}{Remark}

 \newtheorem{corollary}{Corollary}[section]

\definecolor{pink}{rgb}{1,.2,.6}
\definecolor{orange}{rgb}{0.7,0.3,0}
\definecolor{blue}{rgb}{.2,.6,.75}
\definecolor{green}{rgb}{.4,.7,.4}
\definecolor{purple}{RGB}{127,0,255}

\begin{document}
\numberwithin{equation}{section}

\title{Shortest nonzero lattice points in a totally real multi-quadratic Number field and applications}

 \author[Das]{Jishu Das}
 \address{Indian Institute of Science Education and Research Thiruvananthapuram, Maruthamala PO, Vithura,
Thiruvananthapuram - 695551,
 Kerala, India.}
 \email{jishudas@iisertvm.ac.in}

\thanks{}

\date{\today}

\begin{abstract} 
Let $F$ be a multi-quadratic totally real number field. Let $\sigma_1,\dots, \sigma_r$ denote its distinct embeddings. Given $s \in F,$ we give an explicit formula for $\| \sigma(s)\|$ and $\sum_{i<j} \sigma_i(s)\sigma_j(s),$ where $\| \sigma(s)\|=\sqrt{\sum_{i=1}^r(\sigma_i(s))^2}.$ Let $\mathfrak{M}$ be a fractional ideal in $F$ and $\min\left( \mathfrak{M}\right):=\min\{\|\sigma(s)\| \, | \, s \in \mathfrak{M}, s\neq 0 \}.$  The set of shortest nonzero lattice points for  $\mathfrak{M}$  is given by $\{s\in \mathfrak{M} : \| \sigma(s)\|=\min(\mathfrak{M})  \}.$ We provide shortest nonzero lattice points for  $\mathfrak{M}$ in terms of rational solutions to a given Diophantine equation. As an application, we get a refined asymptotic for the Petersson trace formula for the space of Hilbert cusp forms.  We then use the refined asymptotic to obtain a 
 lower bound analogue to Theorem \cite[Theorem 1.6]{JS}.
\end{abstract}
\maketitle


\section{Introduction}
Let $F$ be a totally real number field of degree $r$ and $\sigma_1,\dots, \sigma_r$ denote its distinct embeddings. For a fractional ideal $\mathfrak{M},$ the set $\sigma(\mathfrak{M})$ is a lattice in $\mathbb{R}^r$ being endowed with usual metric given by $\| \sigma(s)\|=\sqrt{\sum_{i=1}^r(\sigma_i(s))^2},$ where  $s\in \mathfrak{M}.$ Let $\min\left( \mathfrak{M}\right):=\min\{\|\sigma(s)\| \, | \, s \in \mathfrak{M}, s\neq 0 \}$ and shortest nonzero lattice points of $\mathfrak{M}$ is given by the set $B_\mathfrak{M}:=\{s\in \mathfrak{M} : \| \sigma(s)\|=\min(\mathfrak{M})  \}.$ A natural questions to ask is:      how does  $B_\mathfrak{M}$  look like.   One of the motivations for this question comes while obtaining a refined version (see Theorem \ref{Main application}) of an asymptotic \cite[Theorem 1]{BDS} for Petersson's trace formula \cite[Theorem 5.11]{KL} for Hibert cusp forms.  To be precise, for an integral ideal $\mathfrak{b}_i\mathfrak{N}$, let us consider equations \eqref{delta_i} and \eqref{A_i} for a fixed index $i$. It can be easily seen that $\delta_i=\min(\mathfrak{b}_i\mathfrak{N})$  and $A_i\subset B_{\mathfrak{b}_i\mathfrak{N}}.$  Thus having the knowledge of  $\min(\mathfrak{b}_i\mathfrak{N})$ and $B_{\mathfrak{b}_i\mathfrak{N}}$ help us to refine the asymptotic.  

 In this paragraph, we discuss the importance of Petersson's trace formula for cusp forms. 
 The discrepancy between two probability measures involving  eigenvalues of cusp forms
on a given closed interval has been a well-studied topic in number theory.
Upper bound estimates (see \cite[Theorem 2]{MSeffective}, \cite[Theorem 6]{MS2} , \cite[Theorem 1.1]{Lau-Li-Wang},  \cite[Theorem 1]{TW}, \cite[Theorem 2]{SZ}) and lower bound estimates (see \cite[Theorem 1.1, Theorem 1.6]{JS}, \cite[Theorem 3]{BDS}, \cite[Theorem 1, Theorem 2]{JD}) for 
discrepancy of two probability measures are of particular interest due to its various
applications to other problems. For upper bounds, the key ingredient is the Eichler-Selberg trace formula (\cite{Sel}) and its estimates. For both upper and lower bounds,  the key ingredient is using estimates for Petersson's trace formula (\cite{PH}).

 Exploring various algebraic objects concerning multi-quadratic fields has interested many mathematicians. For instance,  Balasubramanian, Luca and  Thangadurai give \cite[Theorem 1.1]{Balu} the exact degree of $\Q\left( \sqrt{d_1},\, \dots \,,\sqrt{d_n}\right)$. In 2022, Babu and Mukhopadhyay computed \cite[Theorem 3]{Babu} the explicit structure of the  Galois group of $\Q\left( \sqrt{d_1},\, \dots \,,\sqrt{d_n}\right)$ over $\Q.$ In 1966, Wada \cite{Wada} generalised Kubota's result \cite{Kubota} of finding fundamental units for quartic bicyclic fields by giving an algorithm for computing fundamental units in any given multi-quadratic field. Kuroda \cite{Kuroda} gave the formula for the class number of a multi-quadratic field in terms of class numbers of all its quadratic subfields.  Multi-quadratic fields occur naturally 
as the genus field of imaginary quadratic fields. To be precise, 
given an imaginary quadratic field $K$ with discriminant $d_K, $ let $p_1,\dots , p_n $ be the distinct prime divisors of $d_K.$ The genus field of $K$ is given \cite[Theorem 6.1]{Cox} by $\Q\left( \sqrt{p_1^*}, \dots, \sqrt{p_n^*} \right), $ where $p_i^*=(-1)^{\frac{p_i-1}{2}}p_i$ for all $i.$ The genus field of $K$  help us to find Hilbert class field of $K$ (see \cite[Section 6B]{Cox}). 
In this regard, we would like to obtain shortest nonzero lattice points for a fractional ideal in a  totally real multi-quadratic number field. 
First we obtain an explicit expression for $\|\sigma(s)\|$ given as follows. 
\begin{theorem}\label{mainthmsec1}
Let $d_1,\, \dots \, d_n$ be positive square-free integers such that  $ [\Q\left( \sqrt{d_1},\, \dots\, , \sqrt{d_n}\right):\Q]=2^n$.    Let $s\in \Q\left( \sqrt{d_1},\, \dots \,,\sqrt{d_n}\right)$ be given by
    $$
    a(0)+\sum_{i=1}^n b(i)\sqrt{d_i} +\sum_{i'=2}^n\sum_{j'=1}^{ n \choose i'}   c(i',j') \prod_{\alpha=1}^{i'} \sqrt{d_{s_\alpha}}
    $$
   (see Section \eqref{compute sigma s}). Then 
   $$\| \sigma(s)\| =\sqrt{2^n\left(a(0)^2+\sum_{i=1}^nd_i(b(i))^2  +\sum_{i'=2}^n\sum_{j'=1}^{ n \choose i'}    \left( c(i',j')\right)^2   \prod_{\alpha=1}^{i'} d_{s_\alpha} \right) }     .$$
\end{theorem}
As a corollary to the above theorem, we obtain an expression for trace of $s$ over $\Q$ ($\text{Tr}(s)$) and $\sum_{i<j}\sigma_i(s)\sigma_j(s).$
\begin{corollary}
   Let $s$ be as in Theorem \ref{mainthmsec1}. Then  $\text{Tr}(s)=2^na(0)$ and  $$
\sum_{i<j}\sigma_i(s)\sigma_j(s) =2^{2n-1}a(0)^2 -2^{n-1}\left(a(0)^2+\sum_{i=1}^nd_i(b(i))^2+\sum_{i'=2}^n\sum_{j'=1}^{ n \choose i'}  \left(c(i',j')\right)^2   \prod_{\alpha=1}^{i'} d_{s_\alpha}   \right).
$$    
\end{corollary}
We then classify shortest nonzero lattice points for fractional ideal $\mathfrak{M}$  as rational solutions to a given  Diophantine equation involving $\min\left(\mathfrak{M}\right)$. Precisely,
we obtain the following theorem. 
\begin{theorem}\label{main cor diophan}
    Let $d_1,\, \dots \, d_n$ be positive square-free integers with $ [\Q\left( \sqrt{d_1},\, \dots\, , \sqrt{d_n}\right):\Q]=2^n$. Let   $$S=\left\{ \left(x_0,x_1,\,\dots\, ,x_n,y_{2,1},y_{2,2},\,\dots\,, y_{2,{n \choose 2}},y_{3,1},\, \dots\, ,y_{n,1}\right)  \right\}$$ be the set of  rational solutions of the Diophantine equation 
 \begin{equation}\label{thm 2 dio}
      {2^n}\left(X_0^2+\sum_{i=1}^nd_iX_i^2
    +\sum_{i'=2}^n\sum_{j'=1}^{ n \choose i'}  \left(Y_{i',j'}\right)^2   \prod_{\alpha=1}^{i'} d_{s_\alpha}   
    \right)=(\min\left(\mathfrak{M}\right))^2.
     \end{equation}  
    Then for any  fractional ideal  $\mathfrak{M}$ of  $\Q\left( \sqrt{d_1},\, \dots\, \sqrt{d_n}\right)$,  shortest nonzero lattice points for  $\mathfrak{M}$ are precisely $$\left\{ x_0+ \sum_{i=1}^n x_i\sqrt{d_i}+\sum_{i'=2}^n\sum_{j'=1}^{ n \choose i'}  y_{i',j'}  \prod_{\alpha=1}^{i'} \sqrt{d_{s_\alpha}} \in \mathfrak{M} \,\Bigg|\,   \left(x_0,x_1,\,\dots\, ,x_n,y_{2,1},y_{2,2},\,\dots\,, y_{2,{n \choose 2}},y_{3,1},\, \dots\, ,y_{n,1}\right) \in S \right\}.$$
\end{theorem}
\begin{remark}
    As $\sigma(\mathfrak{M})$ is a discrete subset of $\mathbb{R}^{2^n},$  the Diophantine equation \eqref{thm 2 dio} has at least two solutions.   
\end{remark}
In Sect. \ref{sect3}, we discuss examples of $\min (\mathfrak{M})$ and shortest nonzero lattice points for some specific integral ideals. 
For the Diophantine equation
    $${2^n}\left(X_0^2+\sum_{i=1}^nd_iX_i^2
    +\sum_{i'=2}^n\sum_{j'=1}^{ n \choose i'}  \left(Y_{i',j'}\right)^2   \prod_{\alpha=1}^{i'} d_{s_\alpha}   
    \right)=(\min\left( \mathfrak{M}\right))^2,$$ 
    by a trivial solution, we refer to a solution with at most one nonzero unknown. That is, say if $X_1\neq 0,$ then we must have $X_i=0$ and $Y_{i',j'}=0$ for $2\leq i',i\leq n$  and $1\leq j'\leq {n\choose i'}. $
As an application, we obtain the following refined asymptotic for Petersson's trace formula. We refer to Sect. \ref{secapptraceformula} for notations involved in the following theorem.
\begin{theorem}\label{Main application}
Let $d_1,\, \dots \, d_n$ be square-free integers with $F=Q\left( \sqrt{d_1},\, \dots\, , \sqrt{d_n}\right) $ and  $ [\Q\left( \sqrt{d_1},\, \dots\, , \sqrt{d_n}\right):\Q]=2^n=r$. Suppose $\tilde{x}_i$ be a trivial solution of the Diophantine equation   \begin{equation}\label{Dieqnbin}
{2^n}\left(X_0^2+\sum_{\tilde{i}=1}^nd_{\tilde{i}}X_{\tilde{i}}^2
    +\sum_{i'=2}^n\sum_{j'=1}^{ n \choose i'}  \left(Y_{i',j'}\right)^2   \prod_{\alpha=1}^{i'} d_{s_\alpha}   
    \right)=(\min\left(\mathfrak{b}_i \mathfrak{N}\right))^2
    \end{equation} for some $i$,
    i.e. $2^n\tilde{x}_i^2\tilde{d_i}=(\min\left(\mathfrak{b}_i \mathfrak{N}\right))^2$ where $\tilde{d_i}=\prod_{l=1}^nd^{t_l}_l$ with $t_l\in \{0,1\}$  and $\tilde{d_0}=1.$
    Let $s_i'=\frac{\min\left(\mathfrak{b}_i \mathfrak{N}\right)}{2^{n-1}\sqrt{\tilde{d_i}}}, $
then under the assumptions of Theorem \ref{Main theorem BDS}, as $k_0\rightarrow  \infty$,
$$ 
\frac{e^{2\pi tr_{\mathbb{Q}}^F (m_1+m_2)}}{{\psi(\mathfrak{N})}}
 \Bigg[\prod_{j=1}^r  \frac{(k_j-2)!}{(4\pi \sqrt{\sigma_j(m_1m_2)})^{k_j-1}} \Bigg]
\sum_{{\phi} \in \mathcal{F} }\frac{\lambda_\mathfrak{n}^\phi W_{m_1}^\phi(1) \overline{W_{m_2}^\phi(1)}}{\|\phi \|^2}
 $$
$$=\, \Hat{T}(m_1,m_2,\mathfrak{n})\frac{\sqrt{d_F\Nr(\mathfrak{n})}}{\omega_\mathfrak{N}(m_1/\mathtt{s})\omega_{\mathrm{f}}(\mathtt{s})}
+ \sum_{i=1}^t \sum_{u\in U, \eta_i u\in F^+}\Bigg\{ \omega_{\mathrm{f}}(s_i'\b_i^{-1} ) S_{\omega_\mathfrak{N}} (m_1,m_2;\eta_i u \b_i^{-2};s_i'\b_i^{-1})
 $$$$ 
\frac{\sqrt{\Nr(\eta_i u)}}{\Nr(s_i')}\times  \prod_{j=1}^r\frac{2\pi}{(\sqrt{-1})^{k_j}}J_{k_j-1} \Big(  \frac{4 \pi\sqrt{\sigma_j (\eta_i u m_1m_2 )}}{|\sigma_j(s_i')|}\Big)
 \Bigg\}+ \o\left(\prod_{j=1}^r \big(k_j-1\big)^{-\frac{1}{3}}\right). $$
 Furthermore,  in case of non-existence of a  trivial solution to \eqref{Dieqnbin} for all $i$,  under the assumptions of Theorem \ref{Main theorem BDS}, as $k_0\rightarrow  \infty$,
 $$ 
\frac{e^{2\pi tr_{\mathbb{Q}}^F (m_1+m_2)}}{{\psi(\mathfrak{N})}}
 \Bigg[\prod_{j=1}^r  \frac{(k_j-2)!}{(4\pi \sqrt{\sigma_j(m_1m_2)})^{k_j-1}} \Bigg]
\sum_{{\phi} \in \mathcal{F} }\frac{\lambda_\mathfrak{n}^\phi W_{m_1}^\phi(1) \overline{W_{m_2}^\phi(1)}}{\|\phi \|^2}
 $$
$$=\, \Hat{T}(m_1,m_2,\mathfrak{n})\frac{\sqrt{d_F\Nr(\mathfrak{n})}}{\omega_\mathfrak{N}(m_1/\mathtt{s})\omega_{\mathrm{f}}(\mathtt{s})}+
 \o\left(\prod_{j=1}^r \big(k_j-1\big)^{-\frac{1}{3}}\right). $$
\end{theorem} 
\begin{remark}
    Theorem \ref{Main application} is a refinement of Theorem \ref{Main theorem BDS} in the following manner. We managed to convert the triple sum into a double sum by showing $|A_i|=1,$ whenever equation \eqref{Dieqnbin} corresponding to the index $i$ has a trivial solution. Under the condition of no trivial solution for all $i$, the triple sum altogether vanishes. 
\end{remark}
At the end of Sect. \ref{sect3}, we give two important corollaries of Theorem \ref{Main application}, which are improvement  of \cite[Theorem 2,Theorem 3]{BDS} for a totally real multi-quadratic number field. Corollary \ref{application sato tate lower} is an analogue to \cite[Theorem 1.6]{JS}.
\section{Some algebraic identities and computation of $\| \sigma(s)\|$}
In this section, we prove an algebraic identity (Theorem \ref{S'n}) that explicitly gives $\|\sigma(s)\|$. After that, we classify shortest nonzero lattice points for an fractional ideal $\mathfrak{M}.$ We start with preliminary notations. For this section, let $R$ be a commutative ring with unity.  Let $k=\sum_{j=1}^r t_j \,2^{j-1}$ such that  each $t_i\in \{0,1\}. $ We have one to one correspondence between $k$ and $r$-tuple $(t_1,...\, ,t_r). $ For instance,  $2^r-1$ corresponds to the tuple $(1,1,\dots,1).$  Given $k,$ let $k(j)$ denote the $j$th coordinate of  $(t_1,...\, ,t_r)$ i.e.  $k(j)=t_j$. 
\begin{lemma}\label{Snlemma}
    Let  $a_0,\,\dots,\,a_{n-1}\in R$ with $n\geq 2$. Let $$S_n:=\sum_{k=0}^{2^{n-1}-1}\left(a_0+\sum_{i=1}^{n-1}(-1)^{k(i)}a_i\right)^2,$$ then 
    $$S_n=2^{n-1}\left(a_0^2+\dots +a_{n-1}^2\right).$$
\end{lemma}
\begin{proof}
    Let us denote the coefficient of $a_ia_j$ by $d_{i,j}.$   We show that $d_{i,j}=0$ whenever $i < j.$  First we  consider the case when  $i>0.$  Let $k\in\{1,\dots, 2^{n-1}-1 \}$ and $(t_1,\dots,t_i,\dots,t_j,\dots,t_{n-1})$ be the $(n-1)$-tuple corresponding to $k$. The contribution of $k$ for $d_{i,j}$ is $(-1)^{t_i+t_j}.$ Let $k'$ be an index that corresponds to the tuple $(t_1,\dots,t_i,\dots,t'_j,\dots,t_{n-1})$ where $t'_j=\frac{1+(-1)^{t_j}}{2}.$ The contribution of $k'$ for $c_{i,j} $ is $$(-1)^{t_i+\frac{1+(-1)^{t_j}}{2}}=-(-1)^{t_i+t_j}.$$ Since $k'\in \{1,\dots, 2^{n-1}-1 \}$ with $k'\neq k,$ $d_{i,j}=0.$ Other case (when $i=0$) can be proved  similarly. 
\end{proof}

Let   $s_1,\,\dots,\, s_{i'}$ be a permutation of $i'$ elements taken from the set  $\{1,\,\dots,\,n\} $ such that $s_1<s_2 < \dots < s_{i'} .$ There are $n \choose i'$ many ways of choosing   $s_1,\,\dots,\, s_{i'}$ and  we use the index $j'$ to iterate all those $n \choose i'$ possibilities. For instance, if $i'=n-1,$ then $j'$ varies from $1$ to $n$, where $j'=1$ corresponds to the permutation with $s_t=t$ for $1\leq t \leq n-1.$ Similarly $j'=2$ corresponds to $s_1=1$ and $s_t=t+1$ for $2 \leq t\leq n-1$.   Hence  there is a one to one correspondence between $j'$ and a permutation  $s_1,\,\dots,\, s_{i'}$ with $s_1<s_2 < \dots < s_{i'} .$ Let $a(i),b(i) \in R$ for $0\leq i\leq n.$
We now define a sum 
\begin{equation}\label{S'n}
    S'_n:=\sum_{k=0}^{2^n-1} \left( a(0)+\sum_{i=1}^n (-1)^{k(i)}b(i)a(i) +\sum_{i'=2}^n \sum_{j'=1}^{ n \choose i'} (-1)^{\sum_{\alpha=1}^{i'} k(t_{s_\alpha}) } c(i',j') \prod_{\alpha=1}^{i'} a(s_\alpha)\right)^2, 
\end{equation}
where $c(i',j')\in R$ for $i'\geq 2$, $1\leq j'\leq {n \choose i'}.$
\begin{theorem}\label{S'nlemma}
    Let $S'_n$ be given by the equation \eqref{S'n}. Then 
    \begin{equation}
        S'_n=2^n\left( \left(a(0)\right)^2+\sum_{i=1}^n\left(a(i)b(i)\right)^2+\sum_{i'=2}^n \sum_{j'=1}^{ n \choose i'}    \left(c(i',j')\right)^2 \prod_{\alpha=1}^{i'} (a(s_\alpha))^2\right).
    \end{equation}
\end{theorem}
\begin{proof}
 On adapting the proof of Lemma \eqref{Snlemma},  for some fixed $i,i'$ each greater than $0$, it is not very difficult to see that the coefficient of $a(0)a(i)$ or $a(i)a(i')$, in the expansion of $S'_n$ is $0.$ 
 Now for a fixed $i\geq 0$, fixed $i'\geq 2$ and fixed $j',$ we compute the coefficient of $a(i)\prod_{\alpha=1}^{i'} a(s_\alpha).$ Note that 
 $$
 \sum_{k=0}^{2^n-1}b(i) a(i)\left( (-1)^{\sum_{\alpha=1}^{i'} k(t_{s_\alpha}) } c(i',j')\prod_{\alpha=1}^{i'} a(s_\alpha)\right)=b(i)a(i) c(i',j')\sum_{k=0}^{2^n-1}\left( (-1)^{\sum_{\alpha=1}^{i'} k(t_{s_\alpha}) }  \prod_{\alpha=1}^{i'} a(s_\alpha) \right)=0,
 $$
 where the last step follows from the argument similar to the proof of  Lemma \eqref{Snlemma}.
 
 For $i'\geq 2$ and $\i\geq 2$ fixed and their corresponding $1\leq j' \leq$$n\choose i'  $ and $1\leq \tilde{j}\leq$$ n \choose \tilde{i}  $ fixed, we  focus on computing the coefficient of
 $\left(\prod_{\alpha=1}^{i'} a(s'_\alpha)\prod_{\beta=1}^{\i} a(\tilde{s}_\beta)\right).$ 
 Let $A'=\{s'_1,\,\dots\,,s'_{i'}\}$ and $\tilde{A}=\{\tilde{s}_1,\,\dots\,,\tilde{s}_{\tilde{i}}\}$ and $A=A'\cap \tilde{A}=\{\underline{s}_1,\,\dots\,,\underline{s}_\gamma\}.$  
 Let $A'\setminus A =\{ s''_1,\, \dots \,, s''_{i'-\gamma}\}$ and  $ \tilde{A}\setminus A=\{ \tilde{\tilde{s}}_1,\, \dots \,, \tilde{\tilde{s}}_{\tilde{i}-\gamma}\}.$ 
Now we consider 
\begin{align*}
  &\sum_{k=0}^{2^n-1} \left((-1)^{\sum_{\alpha=1}^{i'} k(t_{s_\alpha}) } \prod_{\alpha=1}^{i'} a(s_\alpha) \right) \left((-1)^{\sum_{\beta=1}^{\i} k(t_{\tilde{s}_\beta}) } \prod_{\alpha=1}^{\i} a(s_\beta) \right)\\
  &=\sum_{k=0}^{2^n-1} \left((-1)^{\left(\sum_{\delta=1}^{\gamma}   2k(t_{\underline{s}_\delta})+\sum_{\alpha=1}^{i'-\gamma}  k(t_{s''_\alpha})+\sum_{\beta=1}^{\i-\gamma} k(t_{\tilde{\tilde{s}}_\beta})\right)  } \prod_{\delta=1}^{\gamma} \left(a(\underline{s}_\delta)\right)^2  \prod_{\alpha=1}^{i'-\gamma} a(s'_\alpha)  \prod_{\beta=1}^{\i-\gamma} a(\tilde{\tilde{s}}_\beta)  \right)\\
   &= \prod_{\delta=1}^{\gamma} \left(a(\underline{s}_\delta)\right)^2 \sum_{k=0}^{2^n-1} \left((-1)^{\left(\sum_{\alpha=1}^{i'-\gamma}  k(t_{s''_\alpha})+\sum_{\beta=1}^{\i-\gamma} k(t_{\tilde{\tilde{s}}_\beta})\right)  } \prod_{\alpha=1}^{i'-\gamma} a(s'_\alpha)  \prod_{\beta=1}^{\i-\gamma} a(\tilde{\tilde{s}}_\beta)  \right)=0.
\end{align*}
Hence we get 
\begin{equation}\label{sumk}
    \sum_{k=0}^{2^n-1} \left((-1)^{\sum_{\alpha=1}^{i'} k(t_{s_\alpha}) } \prod_{\alpha=1}^{i'} a(s_\alpha) \right) \left((-1)^{\sum_{\beta=1}^{\i} k(t_{\tilde{s}_\beta}) } \prod_{\alpha=1}^{\i} a(s_\beta) \right)=0.
\end{equation}
Therefore
\begin{align*}
&\sum_{k=0}^{2^n-1} \left(\sum_{i'=2}^n \sum_{j'=1}^{ n \choose i'}  (-1)^{\sum_{\alpha=1}^{i'} k(t_{s_\alpha}) }   c(i',j') \prod_{\alpha=1}^{i'}  a(s_\alpha) \right) \left( \sum_{\tilde{i}=2}^n    
\sum_{\tilde{j}=1}^{ n \choose \tilde{i}} (-1)^{\sum_{\beta=1}^{\i} k(t_{\tilde{s}_\beta}) }  c(\tilde{i},\tilde{j})\prod_{\alpha=1}^{\i} a(s_\beta) \right)\\
& =\sum_{k=0}^{2^n-1}  
\sum_{i'=2}^n \sum_{j'=1}^{ n \choose i'} \sum_{\tilde{i}=2}^n    \sum_{\tilde{j}=1}^{ n \choose \tilde{i}}     c(i',j') c(\tilde{i},\tilde{j})
\left((-1)^{\sum_{\alpha=1}^{i'}   k(t_{s_\alpha}) } \prod_{\alpha=1}^{i'}  a(s_\alpha) \right) \left((-1)^{\sum_{\beta=1}^{\i} k(t_{\tilde{s}_\beta}) } \prod_{\alpha=1}^{\i} a(s_\beta) \right)\\
&
=\sum_{i'=2}^n \sum_{j'=1}^{ n \choose i'} \sum_{\tilde{i}=2}^n    \sum_{\tilde{j}=1}^{ n \choose \tilde{i}}     c(i',j') c(\tilde{i},\tilde{j}) \left( \sum_{k=0}^{2^n-1}  
\left((-1)^{\sum_{\alpha=1}^{i'} k(t_{s_\alpha}) } \prod_{\alpha=1}^{i'}  a(s_\alpha) \right) \left((-1)^{\sum_{\beta=1}^{\i} k(t_{\tilde{s}_\beta}) } \prod_{\alpha=1}^{\i} a(s_\beta) \right) \right)=0
,
\end{align*}
where the last summation follows from equation \eqref{sumk}.
\end{proof}
\subsection{Computing $\| \sigma(s)\|$  }\label{compute sigma s}
Let  $d_1,\,\dots\,,d_n$ be squarefree natural numbers such that $ [\Q\left( \sqrt{d_1},\, \dots\, , \sqrt{d_n}\right):\Q]=2^n$.  Recall that, for $k=\sum_{j=1}^r t_j \,2^{j-1}$ with $t_i\in \{0,1\}, $ there is one to one correspondence between $k$ and $r$-tuple $(t_1,...\, ,t_r)$ and $k(j)=t_j.$ Note that  that $\left\{ \sqrt{d_1^{k(1)}\cdots d_n^{k(n)} }: 0\leq k\leq 2^n-1 \right\} $ is a basis for $\Q\left( \sqrt{d_1},\, \dots  \, , \sqrt{d_n}\right).$
Also for $i'=2,\dots , n,$ let $s_1,\dots,  s_{i'}$ be a permutation of $\{1,\dots ,n\}. $ Let  $j'$ is chosen in such a manner that it is in correspondence with $s_1,\dots, s_{i'}$ satisfying $s_1<s_2<\dots <s_{i'}$.
Hence any $s \in \Q\left( \sqrt{d_1},\, \dots  \, , \sqrt{d_n}\right)$ can be written as  \begin{equation}\label{write s}
 a(0)+\sum_{i=1}^n b(i)\sqrt{d_i} +\sum_{i'=2}^n\sum_{j'=1}^{ n \choose i'}  c(i',j')  \prod_{\alpha=1}^{i'} \sqrt{d_{s_\alpha}},
\end{equation}
where  $a(0), b(i), c(i',j') \in \Q$.   For $0\leq k\leq 2^n-1$,  it is not very difficult to see that embeddings of $\Q\left( \sqrt{d_1},\, \dots\, , \sqrt{d_n}\right)/\Q$ are given by  \begin{equation}\label{write ss}
\sigma_{k+1}(s) = a(0)+\sum_{i=1}^n (-1)^{k(i)}b(i) \sqrt{d_i} +\sum_{i'=2}^n \sum_{j'=1}^{ n \choose i'} (-1)^{\sum_{\alpha=1}^{i'} k(t_{s_\alpha}) } c(i',j') \prod_{\alpha=1}^{i'} \sqrt{d_{s_\alpha}}     .
\end{equation}
We now compute $\|\sigma(s)\|$ in the  following theorem.
\begin{theorem}\label{Main cor 1}
  Let $d_1,\,\dots\,,d_n$ be squarefree natural numbers such that $ [\Q\left( \sqrt{d_1},\, \dots\, , \sqrt{d_n}\right):\Q]=2^n$. Let $s\in \Q\left( \sqrt{d_1},\, \dots\, ,\sqrt{d_n}\right)$ be given by 
  \begin{equation}s=a(0)+\sum_{i=1}^n b(i)\sqrt{d_i} +\sum_{i'=2}^n\sum_{j'=1}^{ n \choose i'}  c(i',j')  \prod_{\alpha=1}^{i'} \sqrt{d_{s_\alpha}},
  \end{equation} 
  where $b(i), c(i',j') \in \Q$ for all $i, i',j' .$ Then $$\| \sigma(s)\|^2 =2^n\left(a(0)^2+\sum_{i=1}^nd_i(b(i))^2 
  +\sum_{i'=2}^n\sum_{j'=1}^{ n \choose i'}  \left(c(i',j')\right)^2   \prod_{\alpha=1}^{i'} d_{s_\alpha}   \right).$$
\end{theorem}
\begin{proof}
 Note that $\|\sigma(s)\|^2=\sum_{i=1}^{2^n} (\sigma_i(s))^2$ where $\sigma_i(s)$ are given by equation \eqref{write ss}. The proof then follows from Theorem \ref{S'nlemma} on taking $\sqrt{d_i}$ in the place of $a(i)$ for $i\geq 1.$
\end{proof}
\begin{remark}
    We note that $d_i'$s can be negative, this enables us to use Corollary \ref{Main cor 1} for any multi-quadratic extension, not specifically totally real ones. 
\end{remark}
\begin{remark}
    It must be noted that  $[\Q\left( \sqrt{d_1},\, \dots\, , \sqrt{d_n}\right):\Q]$ need to be $2^n$ always. A complete criterion for $[\Q\left( \sqrt{d_1},\, \dots\, , \sqrt{d_n}\right):\Q]$ is given by \cite[Theorem 1.1]{Balu}. For the case of  $[\Q\left( \sqrt{d_1},\, \dots\, , \sqrt{d_n}\right):\Q]=2^l$ with $l<n,$ without loss of generality, we can always choose $d_{i_1}, \,\dots \,, d_{i_l} \in\{ d_1,\dots, d_n\}$  such that $\Q\left( \sqrt{d_1},\, \dots\, , \sqrt{d_n}\right)=\Q\left( \sqrt{d_{i_1}},\, \dots\, , \sqrt{d_{i_l}}\right)$ and use Theorem \ref{Main cor 1}.
\end{remark}
\begin{corollary}
Let  $s$ be as given by equation \eqref{write s}. Let $\text{Tr}(s)$ denote the trace of $s$.  Let $\sigma_1,\dots,\sigma_{2^{n}}$ denote its real embeddings. Then $\text{Tr}(s)=2^na(0)$ and  $$
\sum_{i<j}\sigma_i(s)\sigma_j(s) =2^{2n-1}a(0)^2 -2^{n-1}\left(a(0)^2+\sum_{i=1}^nd_i(b(i))^2+\sum_{i'=2}^n\sum_{j'=1}^{ n \choose i'}  (c(i',j'))^2  \prod_{\alpha=1}^{i'} d_{s_\alpha}\right).
$$    
\end{corollary}
\begin{proof}
Note that \begin{align*}&\text{Tr}(s)=\sum_{i=1}^{2^n} \sigma_i(s)=\sum_{k=0}^{2^{n}-1}\left(a(0)+\sum_{i=1}^n (-1)^{k(i)}b(i) \sqrt{d_i} +\sum_{i'=2}^n \sum_{j'=1}^{ n \choose i'} (-1)^{\sum_{\alpha=1}^{i'} k(t_{s_\alpha}) } c(i',j') \prod_{\alpha=1}^{i'} \sqrt{d_{s_\alpha}}\right)   \\
&
=2^na(0)+ \sum_{i=1}^n b(i) \sqrt{d_i} \left( \sum_{k=0}^{2^{n}-1} (-1)^{k(i)} \right)   +  \sum_{i'=2}^n \sum_{j'=1}^{ n \choose i'}      c(i',j') \prod_{\alpha=1}^{i'} \sqrt{d_{s_\alpha}} \left(  \sum_{k=0}^{2^{n}-1}          (-1)^{\sum_{\alpha=1}^{i'} k(t_{s_\alpha}) }\right)\\
&=2^na(0),
\end{align*}
as for  fixed $i'$ and $j'$, $\sum_{k=0}^{2^{n}-1}          (-1)^{\sum_{\alpha=1}^{i'} k(t_{s_\alpha}) }=0.$  
The computation of $\sum_{i<j}\sigma_i(s)\sigma_j(s)$ then follows from  Theorem \ref{Main cor 1}   and expanding $\left(\sum_{i=1}^{2^n} \sigma_i(s)\right)^2$.
\end{proof}
\begin{proof}[Proof of Theorem \ref{main cor diophan}]
In Theorem \ref{Main cor 1}, on taking  $a(i)=x_i,$ $c(i',j')=y_{i',j'},$ we see that  if  $$x_0+ \sum_{i=1}^n x_i\sqrt{d_i}+\sum_{i'=2}^n\sum_{j'=1}^{ n \choose i'}  y_{i',j'}  \prod_{\alpha=1}^{i'} \sqrt{d_{s_\alpha}}$$ is a shortest non-zero vector for $\mathfrak{M}$,  then $$\left(x_0,x_1,\,\dots\, ,x_n,y_{2,1},y_{2,2},\,\dots\,, y_{2,{n \choose 2}},y_{3,1},\, \dots\, ,y_{n,1}\right)$$ must satisfy the equation \eqref{thm 2 dio}. However, the converse is not necessarily true. Hence we additionally check whether $ x_0+ \sum_{i=1}^n x_i\sqrt{d_i}+\sum_{i'=2}^n\sum_{j'=1}^{ n \choose i'}  y_{i',j'}  \prod_{\alpha=1}^{i'} \sqrt{d_{s_\alpha}} $ lies in $\mathfrak{M}$ or not. 
\end{proof}
\section{Examples and applications}\label{sect3}
In the first subsection, we discuss some examples concerning $\min (\mathfrak{M}) $ and $B_{\mathfrak{M}}$. In the second subsection, we apply Theorem \ref{main cor diophan} to obtain a refined asymptotic for Petersson's trace formula and discuss some immediate applications of it. 
\subsection{Examples}
 We start with  a lemma concerning lower bound for $\| \sigma(s)\|$ which will be primarily used in this section. 
\begin{lemma}\label{sigma(s)>=r0.5nm(s)}
      Let $F$ be a number field of degree $r$. Then for $s\in F,$  we have \begin{align*}\|\sigma(s)\|\geq \sqrt{r} \left|\text{Nm}(s)\right|^{\frac{1}{r}}.\end{align*}
 \end{lemma}
\begin{proof}
    Using Cauchy-Schwartz inequality, we have 
$$ \sqrt{(1^2+...+1^2)( \sigma_1^2(s)+ ... +\sigma_r^2(s))}\geq \left|\sigma_1(s)\right|+...+\left|\sigma_r(s)\right|. $$ 
 AM GM inequality implies $$ \frac{\left|\sigma_1(s)\right|+...+\left|\sigma_r(s)\right|}{r}\geq \left(\left|\sigma_1(s)\right|\times...\times \left|\sigma_r(s)\right|\right)^{\frac{1}{r}}=\left|\text{Nm}(s)\right|^{\frac{1}{r}}.$$
    Combining both, we obtain 
    \begin{equation*}
      \|\sigma(s)\|\geq \sqrt{r} \left|\text{Nm}(s)\right|^{\frac{1}{r}}.
    \end{equation*}
\end{proof}
     \begin{corollary}\label{examplecorollary}
   Consider an ideal $\mathfrak{M}=\langle s \rangle$ in $\Q\left( \sqrt{d_1},\, \dots\, , \sqrt{d_n}\right)$ with  $ [\Q\left( \sqrt{d_1},\, \dots\, , \sqrt{d_n}\right):\Q]=2^n$. Then   $\min(\mathfrak{M})\geq 2^{n-1} \left|  \text{Nm}(s) \right|^{\frac{1}{2^n}}.$
\end{corollary}
\begin{proof}
    Let $s' \in \mathfrak{M},$ then $s'=st$ which implies $\left|\text{Nm}(s')\right|\geq \left| \text{Nm}(s)\right|.$ Hence $\| \sigma(s') \|\geq 2^{n-1} \left|  \text{Nm}(s) \right|^{\frac{1}{2^n}}.$ Taking minimum over $s'\in \mathfrak{M}$ with $s'\neq 0$ proves the claim. 
\end{proof}
\begin{example}
    Consider an integral ideal $\mathfrak{M}=\langle s \rangle$ in $\Q\left( \sqrt{d_1},\, \dots\, , \sqrt{d_n}\right)$ with  $ [\Q\left( \sqrt{d_1},\, \dots\, , \sqrt{d_n}\right):\Q]=2^n$. Let $\mathcal{O}^\times$ denote the unit group of its ring of integers.  Let $s\in \Q\left( \sqrt{d_1},\, \dots\, ,\sqrt{d_n}\right)$ be given by $$s=a(0)+\sum_{i=1}^n b(i)\sqrt{d_i} +\sum_{i'=2}^n\sum_{j'=1}^{ n \choose i'} c(i',j')   \prod_{\alpha=1}^{i'} \sqrt{d_{s_\alpha}}$$
    and $\max(s)=\text{o}\left(a(0)\right),$ where $$\max(s)=\max\left( \{|b(i)|: i=1, \,\dots\, n \} \cup \{|\sqrt{d}_i|: i=1, \,\dots\, n \}  \cup \left\{|c(i',j')|: i'=2, \,\dots\, n, j'=1, \dots,  {n \choose i'} \right\}   \right).$$ Then $$\min\left( \mathfrak{M} \right)=\min\left(  \{ \|\sigma (us)\| :  u \in \mathcal{O}^\times  \} \right).$$
\end{example}
\begin{proof}
 Let $s' \in \mathfrak{M} $ with $\|\sigma(s')\|=\min (\mathfrak{M}) $ and  suppose $|\text{Nm}(s')|\geq 2 |\text{Nm}(s)|.$ Using Lemma \ref{sigma(s)>=r0.5nm(s)},  \begin{equation}\label{exampfeqn}
 \|\sigma(s)\|\geq \|\sigma(s')\| \geq  2^{\frac{n}{2}+\frac{1}{2^n}} |\text{Nm}(s)|^{\frac{1}{2^n}}.
 \end{equation}
 Using  Theorem \ref{Main cor 1} and the condition $\max(s)=a(0)+\text{o}\left(a(0)\right),$ we get  $\|\sigma(s)\|=2^{\frac{n}{2}}a(0)+o(a(0)).$
 The given condition further implies $|\text{Nm}(s)|=\left(a(0)\right)^{2^n}+o(a(0)). $ However, equation \eqref{exampfeqn} implies 
 $$
 2^{\frac{n}{2}}a(0)+o(a(0))\geq 2^{\frac{1}{2^n}} 2^{\frac{n}{2}} \left(\left(a(0)\right)^{2^n}+o(a(0))\right)^{\frac{1}{2^n}},
 $$
 an impossibility as $2^{\frac{1}{2^n}}>1 $ for a fixed $n.$ We must have $|\text{Nm}(s')|=  |\text{Nm}(s)|,$ which proves the claim.  
\end{proof}
\begin{remark}
   To find $\min(\mathfrak{M})$ with $\mathfrak{M}$ satisfying the assumptions given in the above example, it is sufficient to minimise $\| \sigma(us) \|$, where $u$ varies over $\mathcal{O}^\times$. For explicit computations of unit groups of  $\Q\left( \sqrt{d_1},\, \dots\, ,\sqrt{d_n}\right),$ one may refer to \cite{Wada}, \cite{Benjamin}, \cite{CEMM} and \cite{CEMMZ}. 
\end{remark}
For an ideal $\mathfrak{M},$ recall that  $B_{\mathfrak{M}}$ denote shortest nonzero lattice points of $\sigma(\mathfrak{M}),$ i.e. $B_{\mathfrak{M}}=\{ s \in \mathfrak{M}: \|\sigma(s)\|=\min (\mathfrak{M}) \}.$
\begin{example}
   Consider the    integral ideal  $\left\langle n+\sqrt{3} \right\rangle $ in $\mathbb{Q}\left(\sqrt{3}\right)$ and $n\geq 4.$ Then $\min\left(\langle n+\sqrt{3}\rangle \right)=2n^2+6$ and $B_{ \left\langle n+\sqrt{3} \right\rangle}=\{  n+\sqrt{3}, -n-\sqrt{3} \}.$
    \end{example}
    \begin{proof}
       Let $t=c+f\sqrt{3}$ be such that $\|\sigma(t)\|=\min \left(\langle n+\sqrt{3}\rangle \right).$ Suppose $t \in \left\langle n+\sqrt{3}\right\rangle $ be such that $|\text{Nm}(t)|\geq 2 \text{Nm}\left(n+\sqrt{3}\right).$ On applying Lemma \ref{sigma(s)>=r0.5nm(s)}, 
       \begin{align*}
     \sqrt{2n^2+6} \geq  \|\sigma(t)\|\geq \sqrt{2|\text{Nm}(t)|}\geq 2\sqrt{n^2-3}.
       \end{align*}
       This forces $2n^2+6\geq 4n^2-12$, an impossibility, as $n\geq 4.$ Hence $|\text{Nm}(t)|=\text{Nm}\left(n+\sqrt{3}\right)$ which implies $|c^2-3f^2|=n^2-3.$ We have $c^2+3f^2\leq n^2+3.$ 
       
       Suppose $c^2+3f^2\leq n^2+2,$ then $c^2+3f^2-|c^2-3f^2|\leq 5.$ We must have $2c^2\leq 5$ or $6f^2\leq 5.$ This implies either  $c=1,0$ or $f=0.$        
       For the case $c=0,$ we get $3f^2=n^2-3,$ which implies $3|n. $ On writing $n=3n'$ yields $f^2=3n'^2-1.$ On going modulo $3$, we get $f^2\equiv 2$ mod $3,$ an absurdity. Now we focus when $c=1.$ For such thing to happen  $3f^2-1$ must be equal to $n^2-3.$ Again going modulo $3$ yields $n^2\equiv 2$ mod 3, a contradiction. The case of $f=0$ is not possible as $n^2-3$ is never a perfect square for $n\geq 4.$  
       
       This forces $c^2+3f^2=n^2+3$ and $|c^2-3f^2|=n^2-3.$ The scenario of $c^2-3f^2=3-n^2$ would imply $c^2=3,$ an absurdity.  Hence only possible pairs of $(c,f)$ possible belong to $$\{(-n,-1),(n,1),(n,-1),(-n,1) \}.$$ Suppose $n-\sqrt{3} \in \left\langle n+\sqrt{3}\right\rangle, $ then $2\sqrt{3}\in \left\langle n+\sqrt{3}\right\rangle.$ Consequently $(n^2-3)|36$ which is not possible as $n\geq 4. $ Thus $B_{ \left\langle  n+\sqrt{3} \right\rangle}=\{ n+\sqrt{3},  -n-\sqrt{3} \}.$  
    \end{proof} 
\begin{example}
In this example, we consider a generic integral ideal of a real quadratic field generated by a single element.
  Let us consider a real quadratic field $\mathbb{Q}(\sqrt{d})$  with $d\equiv 2,3$ (mod $4$) and an integral ideal   $\mathfrak{M}$ generated by $ a+b\sqrt{d}$ with $a,b,\in \mathbb{Z}.$ For any $c,f \in \mathbb{Z}$,
       \begin{align*}
       &\left\| \sigma\left(\left(a+b\sqrt{d}\right)\left(c+f\sqrt{d}\right)\right)\right\|=\left\| \sigma\left( (ac+bdf)+(bc+af)\sqrt{d} \right)\right\|\\
       &=\sqrt{2\left( \left(ac+bdf\right)^2+d\left(bc+af\right)^2\right)}=\sqrt{\left(2a^2+2db^2\right)\left( c^2+df^2\right)+\left(8abd\right)cf}.
       \end{align*}
       The discriminant of the quadratic form $\left(2a^2+2db^2\right)\left( c^2+df^2\right)+\left(8abd\right)cf$ is given by 
       $$\left(8abd\right)^2- 4\left(2a^2\right)\left(2d^2b^2\right)=48a^2b^2d^2.$$ Hence the quadratic form is indefinite and its minimum can be computed using  \cite[Theorem A(iii), Theorem 1]{Peter}. This helps us to obtain $\min\left(\langle a+b\sqrt{d}\rangle \right).$
       
Now  let $c_0+f_0\sqrt{d}$ be such that $\left\| \sigma(c_0+f_0\sqrt{d})\right\|=\min\left( \langle a+b\sqrt{d}\rangle \right)$. We can assume $(c_0,d_0)=1,$ since on supposing  $(c_0,d_0)=g>1,$  $\left\|\sigma(c_0+f_0\sqrt{d})\right\| > \left\|\sigma\left(\frac{c_0}{g}+\frac{f_0}{g}\sqrt{d}\right)\right\|$, a contradiction to minimality.  To solve the Diophantine equation   $2c_0^2+2df_0^2-(\min\left(\mathfrak{M}\right))^2=0,$ one can refer to the classic book by Mordell \cite[Chapter 19]{Mordell}. For the scenario of  $\frac{(\min\left(\mathfrak{M}\right))^2}{2}=p,$ for a prime integer $p$, one may refer to  \cite[Theorem 5.26]{Cox}.  
\end{example}
\begin{example}[{{\cite[Lemma 12]{BDS}}}]
    Let us consider the ideal  $\langle n\rangle  $ in a totally real field $F$ with $n \in \mathbb{Z} $ and $[F:Q]=r. $ Then $\min\left( \langle n\rangle\right)=|n|\sqrt{r}$  and $B_{\langle n \rangle} = \{n, -n\}.$
\end{example}
\subsection{An application to Petersson trace formula}\label{secapptraceformula}
Petersson's trace formula for cusp forms gives us a weighted trace for   Hecke operators (see for instance  \cite[Equation (2.1)]{JD}). For classical cusp forms, one can refer to \cite[Section 2.1.3]{JS}. Recently, Jung and Sardari \cite[Theorem 1.7]{JS}, Das \cite[Theorem 4]{JD} have used asymptotic for Petersson's trace formula to obtain a lower bound for the discrepancy in the Sato-Tate measure (see \cite[Theorem 1.6]{JS}, \cite[Theorem 1]{JD}).   

 Now we focus our attention on Hilbert cusp forms.  
Let $F $ be a totally real number field and $r=[F :\mathbb{Q}].$
Let $\nu=\nu_{\gp}$ be the discrete valuation corresponding to a prime ideal $\gp$. Let $F_\nu$ be the completion of $F$ with respect to the valuation $\nu$. Let $\mathcal{O}_\nu $ be the ring of integers of the local field $F_\nu$ and  $\hat{\mathcal{O}}=\prod_{\nu<\infty} \mathcal{O}_\nu.$ 
Let $\A$ denote the ad\`ele ring of $F$ with finite ad\`eles $\A_{f}$, so that $\A=F_{\infty} \times \A_{f}$ where $F_\infty =F \otimes \mathbb{R} \cong \mathbb{R}^r$.
For a fractional ideal $\mathfrak{a},$ let $\Hat{\mathfrak{a}}=\prod_{\nu<\infty} \mathfrak{a}_\nu$ where  $\mathfrak{a}_\nu$ denote its localisation.

Let $\mathfrak{N}$ be an integral ideal of $\mathcal{O}$. Let $k=(k_1,\,...\,,k_r) $ be an $r$-tuple of integers with $k_j> 2$. Let $\omega : F^\times \backslash \A^\times \rightarrow \mathbb{C}^\times$ be a unitary Hecke character. We can decompose $\omega$ as a product of local characters, $\omega=\prod_\nu \omega_\nu$ where $\omega_\nu: F_\nu^\times \to \mathbb C^\times$ are the local characters. Let $\omega_{f}=\prod_{\nu<\infty} \omega_\nu $ We further assume that
\begin{enumerate}
\item the conductor of $\omega$ divides $\mathfrak{N}$,
\item $\omega_{\infty_j}(x)=\mathrm{sgn} (x)^{k_j} $ for all $j=1,\dots, r$.
\end{enumerate}
The first condition means that $\omega_\nu $ is trivial on $1+\mathfrak{N} \mathcal{O}_\nu$ for all $\nu | \mathfrak{N},$ and unramified for all $\nu \nmid \mathfrak{N} .$
Let $K_{f}=\prod_{\nu<\infty} \GL_2(\mathcal{O}_\nu)$ be the standard maximal compact subgroup of $\GL_2 (\A_{f}) .$
Let 
\begin{equation}\label{K_1-def}
K_1(\mathfrak{N})=  \left\{ \begin{pmatrix} a & b\\  c & d \end{pmatrix} \in K_{f}  : c \in \mathfrak{N} \Hat{\mathcal{O}}, d\in 1+ \mathfrak{N}\Hat{\mathcal{O}} \right\},
\end{equation} and let   $A_k(\mathfrak{N},\omega) $ be the  space  of Hilbert cusp forms with respect to  $K_1(\mathfrak{N})$ (see \cite[Proposition 3.1]{KL} for definition of $A_k(\mathfrak{N},\omega)$), of weight $k$ and  unitary Hecke character $\omega$.

 In 2008, Knightly and Li gave the Petesson's trace formula for the space $A_k(\mathfrak{N},\omega).$  For convenience, we recall the formula after some more notations. 
\begin{itemize}
    \item Let $\sigma_1, \dots, \sigma_r$ be the embeddings of $F$ into $\R$ and let $\sigma = (\sigma_1, \dots, \sigma_r) : F \to \R^r$.  
\item Let $\mathfrak{n}$ and $\mathfrak{N}$ be ideals in $\mathcal O$ such that $( \mathfrak{n}, \mathfrak{N}) = 1$.  
\item Let $F^{+}$ denote the set of totally positive elements of $F$.
Let $d_F$ denote the discriminant of $F$.
\item Let $\mathcal{O}^\times$ denote the unit group of $F$, and let $U$ be a fixed set of representatives for $\mathcal{O}^\times/{\mathcal{O}^\times}^2$.  $U$ is a finite set, and $|U| = 2^r$.
\item Let $N' :F\rightarrow \mathbb{Q}$ denote the norm map. For a nonzero ideal $\mathfrak{a}\subset \mathcal{O} $, let $\Nr(\mathfrak{a})=|\mathcal{O}/\mathfrak{a}|.$ For $\alpha \in F^\times,$ we have $$\Nr(\alpha):=\Nr(\alpha\mathcal{O})=|N'(\alpha)|.$$ 
\item  Let $\mathfrak{d} ^{-1}=\{x\in F \, : \,  \mathrm{tr}_{\mathbb{Q}}^{F}  (x)\subset \mathbb{Z}  \}$
denote the inverse different, where $\mathrm{tr}_{\mathbb{Q}}^{F}  (x)$ denotes the trace of $x\in F$. We also let $\mathfrak{d}_+^{-1}=\mathfrak{d}^{-1}\cap F^+$.
\item For $m \in \mathfrak{d}^{-1}_{+}, W^\phi_{m}(1)$ are Whittaker function evaluated at $1$ (see \cite[Equation 10]{KL}).
\end{itemize}
We refer the reader to  \cite[Section 4.3]{KL} for the definition of the Hecke operators $T_{\mathfrak{n}}$.
\begin{theorem}[{{\cite[Thm.\ 5.11]{KL}}}]\label{Petersson-Trace-Formula}
Let $\mathfrak{n}$ and $\mathfrak{N}$ be integral ideals with $(\mathfrak{n},\mathfrak{N})=1.$ Let $k=(k_1,...,k_r)$ with all $k_j>2.$ Let $\mathcal{F}$ be an orthogonal basis for $A_{k}(\mathfrak{N},\omega)$ consisting of eigenfunctions for the Hecke operator $T_\mathfrak{n}.$ Then for any $m_1,m_2\in \mathfrak{d}_+^{-1}$, we have 
\begin{multline}  \label{PTF}
\frac{e^{2\pi \mathrm{Tr}_{\mathbb{Q}}^F (m_1+m_2)}}{{\psi(\mathfrak{N})}} \Bigg[\prod_{j=1}^r  \frac{(k_j-2)!}{(4\pi \sqrt{\sigma_j(m_1m_2)})^{k_j-1}}\Bigg] \sum_{{\phi}\in \mathcal{F} } \frac{\lambda^\phi_\mathfrak{n} W_{m_1}^\phi(1) \overline{W_{m_2}^\phi(1)}}{\| \phi \|^2} 
=  \Hat{T}(m_1,m_2,\mathfrak{n}) \frac{\sqrt{d_F\Nr(\mathfrak{n})}}{\omega_\mathfrak{N}(m_1/\mathtt{s})\omega_{f}(\mathtt{s})} \\
+ \sum_{i=1}^t\sum_{\substack{u\in U \\ \eta_i u\in F^+}} \sum_{\substack{s\in \mathfrak{b_i}\mathfrak{N}/\pm \\ s\neq 0 }} \Bigg\{ \omega_{f}(s\mathtt{b}_i^{-1} ) S_{\omega_\mathfrak{N}} (m_1,m_2;\eta_i u \mathtt{b}_i^{-2};s\mathtt{b}_i^{-1})  \\
 \times  \frac{\sqrt{\Nr(\eta_i u)}}{\Nr(s)} \times\prod_{j=1}^r\frac{2\pi}{(\sqrt{-1})^{k_j}}J_{k_j-1} \left(  \frac{4 \pi\sqrt{\sigma_j (\eta_i u m_1m_2 )}}{|\sigma_j(s)|} \right) \Bigg\}.
\end{multline}  
\begin{itemize}
     \item where $\Hat{T}(m_1,m_2,\mathfrak{n})\in \{0,1\}$ is non zero if and only if there exists $\mathtt{s}\in \Hat{\mathfrak{d}}^{-1}$ such that $m_1, m_2
\in  \mathtt{s}\Hat{\mathcal{O}}   $ and $m_1m_2 \Hat{\mathcal{O}}=\mathtt{s}^2\Hat{\mathfrak{n}}$,
        \item $U$ is a set of representative for $\mathcal{O}^\times/{\mathcal{O}^\times}^2$,
        \item  For the equation  $[\mathfrak{b}]^2[\mathfrak{n}]=1$ in the ideal class group, let  integral ideals $\mathfrak{b}_i$s are a set of representatives for  $i=1,\,\dots\,,t$. 
 Let $\mathtt{b_i}\Hat{\mathcal{O}}=\Hat{\mathfrak{b}}_i$ for all $i.$
        \item $\eta_i\in F$ generates the principal ideal $\mathfrak{b}_i^2\mathfrak{n}$,
        \item $\omega_\mathfrak{N}=\prod_{v|\mathfrak{N}}\omega_v \prod_{v \nmid \mathfrak{N}} 1$,
        \item and $\psi(\mathfrak{N})=[K_{f} : K_0(\mathfrak{N})]=\Nr(\mathfrak{N})\prod_{\mathfrak{p}|\mathfrak{N}}\Big(1+\frac{1}{\Nr(\mathfrak{p})}\Big)$.
    \end{itemize}
   \end{theorem} 
 
For a fixed level $\mathfrak{N}, $ \cite[Theorem 1]{BDS} gives  an asymptotic for the above trace formula as $k_0\rightarrow\infty$   where $k_0=\min(k_1,\,...\,,k_r)$.  
We discuss some notations to be used further. 
\begin{itemize}
\item Let us consider 
\begin{align}\label{delta_i}
\delta_i =\inf\{\|\sigma(s)\|  \,:\, s\in \mathfrak{b_i}\mathfrak{N}/\pm, s\neq 0 \},
\end{align}
where each $\mathfrak{b_i}$ is as defined in Theorem \ref{Petersson-Trace-Formula}. Note that $\min( \mathfrak{b_i}\mathfrak{N})=\inf\{\|\sigma(s)\|  \,:\, s\in \mathfrak{b_i}\mathfrak{N}/\pm, s\neq 0 \}.$
\item We take $\tilde{\delta}_i=\frac{\delta_i}{2\sqrt{r}}$ and 
let   
\begin{align}\label{A_i}
A_i=\cap_{j=1}^r\{  s\in \mathfrak{b_i}\mathfrak{N}/\pm \,:\, |\sigma_j(s)|\leq  2 \delta_i, s\neq 0\}.
\end{align} 
For a fixed $i,$ $A_i$ is a discrete bounded set in $\R^r.$  This makes $A_i $ finite for each $i$.   
\item We define
$$\gamma_{j}=\max \left \{\sqrt{\sigma_{j}(\eta_i u)} \,| \,i =1,\,...\,,t  ,u\in U, \eta_i u\in F^+\right \} .$$  
\end{itemize}
\begin{theorem}[{{\cite[Theorem 1]{BDS}}}]\label{Main theorem BDS}
Let $\mathfrak{N}$ and $\mathfrak{n}$ be fixed integral ideals in $F$ such that $( \mathfrak{n}, \mathfrak{N}) = 1$.  Let $\omega:\,F^{\times} \backslash \mathbb A^{\times} \to \mathbb C^{\times}$ be a unitary Hecke character. Let the conductor of $\omega$ divide $\mathfrak{N}$ and $\omega_{\infty_j}(x)=\mathrm{sgn} (x)^{k_j} $ for all $j=1,\dots, r$.   For $k = (k_1,k_2,\dots, k_r)$ with all $k_j >2$, let $A_k(\mathfrak{N},\omega)$ denote the space of Hilbert cusp forms of weight $k$ and character $\omega$ with respect to $K_1(\mathfrak{N})$ and let $\mathcal{F}$ be an orthogonal basis for $A_{k}(\mathfrak{N},\omega)$ consisting of eigenfunctions of the Hecke operator $T_{\mathfrak{n}}$.  Let $A_i$'s and $\tilde{\delta}_i$  be as defined above.  Let $k_0=\min \{ k_j \,  | \, j\leq r\}$, and suppose $m_1,m_2 \in \mathfrak{d}^{-1}_+$ satisfy
$$
\frac{2\pi \gamma_j\sqrt{\sigma_{j}(m_1m_2)}}{\tilde{\delta}_i}\in \Big((k_{j}  -1)-(k_{j}  -1)^{\frac{1}{3}},(k_{j}  -1)\Big) \text{ for all }j \leq r.
$$ 
Then,  as $k_0\rightarrow  \infty$,
$$ 
\frac{e^{2\pi tr_{\mathbb{Q}}^F (m_1+m_2)}}{{\psi(\mathfrak{N})}}
 \Bigg[\prod_{j=1}^r  \frac{(k_j-2)!}{(4\pi \sqrt{\sigma_j(m_1m_2)})^{k_j-1}} \Bigg]
\sum_{{\phi} \in \mathcal{F} }\frac{\lambda_\mathfrak{n}^\phi W_{m_1}^\phi(1) \overline{W_{m_2}^\phi(1)}}{\|\phi \|^2}
 $$
$$=\, \Hat{T}(m_1,m_2,\mathfrak{n})\frac{\sqrt{d_F\Nr(\mathfrak{n})}}{\omega_\mathfrak{N}(m_1/\mathtt{s})\omega_{\mathrm{f}}(\mathtt{s})}
+ \sum_{i=1}^t \sum_{u\in U, \eta_i u\in F^+}\sum_{s\in A_i }\Bigg\{ \omega_{\mathrm{f}}(s\b_i^{-1} ) S_{\omega_\mathfrak{N}} (m_1,m_2;\eta_i u \b_i^{-2};s\b_i^{-1})
 $$$$ 
\frac{\sqrt{\Nr(\eta_i u)}}{\Nr(s)}\times  \prod_{j=1}^r\frac{2\pi}{(\sqrt{-1})^{k_j}}J_{k_j-1} \Big(  \frac{4 \pi\sqrt{\sigma_j (\eta_i u m_1m_2 )}}{|\sigma_j(s)|}\Big)
 \Bigg\}+ \o\left(\prod_{j=1}^r \big(k_j-1\big)^{-\frac{1}{3}}\right). $$ 
\end{theorem}
 Theorem \ref{Main theorem BDS} is a generalisation of the asymptotic \cite[Theorem 1.7]{JS} for the Petersson's trace formula for classical cusp forms. 
We now  prove the following  refined version  of Theorem \ref{Main theorem BDS} for multi-quadratic number fields.
\begin{proof}[Proof of Theorem \ref{Main application}]
    We have $A_i=\cap_{j=1}^r\{  s\in \mathfrak{b_i}\mathfrak{N}/\pm \,:\, |\sigma_j(s)|\leq     \frac{\delta_i}{\sqrt{r}} , s\neq 0\},$ where $\delta_i=\min(\mathfrak{b}_i\mathfrak{N})$.   Note that \begin{equation}\label{equalityholds}
    \|\sigma(s) \|\leq \sqrt{\sum_{i=1}^r \frac{\delta_i^2}{r}}=\delta_i
     \end{equation} 
    for $s \in A_i.$ 
    However, by choice of $\delta_i,$ we must have equality in equation \eqref{equalityholds} which implies  $A_i\subset B_{\mathfrak{b_i}\mathfrak{N}}.$  Furthermore  we have $|\sigma_j(s)|=\frac{\delta_i}{\sqrt{r}}$ for all $j=1,\dots,r.$ By an argument similar to \cite[Lemma 10]{BDS}, we get $s=\frac{\delta_i}{\sqrt{r}}.$ Thus $s=\sqrt{\tilde{d}}$ for some rational number $\tilde{d}$.   
     Now using Theorem \ref{main cor diophan}, it follows that $s$ must correspond to a trivial solution to the equation \eqref{Dieqnbin}. 
     In the case of non-existence of a trivial solution to the equation \eqref{Dieqnbin} for some index $i$, we can not have $|\sigma_j(s)|=\frac{\delta_i}{\sqrt{r}}$ for all $j=1,\dots,r.$ This makes the set $A_i$ empty and the proof is complete. 
\end{proof}
Theorem \ref{Main application} allows us to obtain the following corollary, which is a variant of \cite[Theorem 2]{BDS}.
\begin{corollary}\label{Single term}
 Let $F$ have odd narrow class number. Let $d_1,\, \dots \, d_n$ be square-free integers with $F=Q\left( \sqrt{d_1},\, \dots\, , \sqrt{d_n}\right) $ and  $ [\Q\left( \sqrt{d_1},\, \dots\, , \sqrt{d_n}\right):\Q]=2^n=r$.  Further let $\mathfrak{b}_1\mathfrak{N}=\tilde{s}\mathcal{O},$ where $\tilde{s}=a\sqrt{\tilde{d}}$, $\tilde{d}=\prod_{l=1}^{t_l} d_l^{\tilde{t}_l}$ with $\tilde{t}_l\in\{0,1\}$ and $a\in \mathbb{Z}$. Under the assumptions of Theorem \ref{Main theorem BDS} and $S_{\omega_\mathfrak{N}} (m_1,m_2;\eta_1 \b_1^{-2};\tilde{s}\b_1^{-1})\neq 0$ for $m_1$ and $m_2$,
  as $k_0\rightarrow  \infty$,
\begin{align*}
 \Bigg| \frac{e^{2\pi tr_{\mathbb{Q}}^F (m_1+m_2)}}{{\psi(\mathfrak{N})}} \prod_{j=1}^r  \frac{(k_j-2)!}{(4\pi \sqrt{|\sigma_j(m_1m_2)|})^{k_j-1}} 
& \sum_{{\phi}\in \mathcal{F} }\frac{\lambda_\mathfrak{n}^\phi W_{m_1}^\phi(1) \overline{W_{m_2}^\phi(1)}}{\|\phi\|^2} - \Hat{T}(m_1,m_2,\mathfrak{n})\frac{\sqrt{d_F\Nr(\mathfrak{n})}}{\omega_\mathfrak{N}(m_1/\mathtt{s})\omega_{\mathrm{f}}(\mathtt{s})}\Bigg| \\ 
 &\gg \prod_{j=1}^r (k_j-1)^{-\frac{1}{3}}.
\end{align*}
\end{corollary}
\begin{proof}
We have $t=1$ and $\left|\{u \,| \, u\in U, \eta_1 u\in F^+\}\right|=1$ since narrow class number of $F$ is odd (see \cite[Lemma 9]{BDS}).
 Using Corollary \ref{examplecorollary}, we get $$\min\left(\mathfrak{b}_1\mathfrak{N}\right)\geq 2^{n-1} \left|\text{Nm}\left(|a|\sqrt{\tilde{d}}\right)\right|^{\frac{1}{2^n}}=2^{n-1}|a|\sqrt{\tilde{d}}.$$
 However $\|\sigma(\tilde{s})\|=2^{n-1}|a|\sqrt{\tilde{d}}$ by Theorem \ref{Main cor 1},
which implies $\min(\mathfrak{b}_1\mathfrak{N})=2^{n-1}|a|\sqrt{\tilde{d}}.$ This implies the Diophantine equation given by  \eqref{Dieqnbin} corresponding to $\mathfrak{b}_1\mathfrak{N}$ has a trivial solution. This ensures we can apply the first part of Theorem \ref{Main application}.  We see that the lower bound of $\prod_{j=1}^r (k_j-1)^{-\frac{1}{3}}$ holds after proceeding similar to the proof of  \cite[Corollary 14]{BDS}.
\end{proof}
Theorem \ref{Main application} also talks about a scenario, where it is not possible to get a corollary like the above. To be precise, suppose for all $i$, there is no trivial solution for equation \eqref{Dieqnbin},  then it is not possible to apply Theorem \ref{Main application} to obtain Corollary \ref{Single term}, because
$$ 
\frac{e^{2\pi tr_{\mathbb{Q}}^F (m_1+m_2)}}{{\psi(\mathfrak{N})}}
 \Bigg[\prod_{j=1}^r  \frac{(k_j-2)!}{(4\pi \sqrt{\sigma_j(m_1m_2)})^{k_j-1}} \Bigg]
\sum_{{\phi} \in \mathcal{F} }\frac{\lambda_\mathfrak{n}^\phi W_{m_1}^\phi(1) \overline{W_{m_2}^\phi(1)}}{\|\phi \|^2}
$$$$ -\, \Hat{T}(m_1,m_2,\mathfrak{n})\frac{\sqrt{d_F\Nr(\mathfrak{n})}}{\omega_\mathfrak{N}(m_1/\mathtt{s})\omega_{\mathrm{f}}(\mathtt{s})}=
 \o\left(\prod_{j=1}^r \big(k_j-1\big)^{-\frac{1}{3}}\right). $$

We now consider another application of Theorem \ref{Main application} to obtain a lower bound for the discrepancy between the Sato-Tate measure and a given discrete measure. This is motivated by the application of \cite[Theorem 1.6]{JS}  to obtain a lower bound like \cite[Theorem 1.6]{JS} for the classical setting of $F=\Q$. For a prime ideal $\mathfrak{p},$ let  
$\kappa^\phi_{\mathfrak{p}^l}=\frac{\lambda^\phi_{\mathfrak{p}^l}}{\sqrt{\Nr(\mathfrak{p}^l)}}$ for a natural number $l.$ We have $\kappa^\phi_{\mathfrak{p}^l}\in \mathbb{R}$  and by the Ramanujan conjecture $\kappa^\phi_{\mathfrak{p}}\in [-2,2].$ Let $X_l$ denote Chebyshev polynomial of degree $l$, then by \cite[Proposition 4.5]{KL}, $X_l\left(\kappa^\phi_{\mathfrak{p}}\right)=\kappa^\phi_{\mathfrak{p}^l}.$
Let us consider the Sato-Tate measure defined  by
$$ \mu_\infty(x)=\frac{1}{\pi}\sqrt{1-\frac{x^2}{4}} ,
$$
 for $x\in [-2,2].$ 
Define a weight $w_\phi=\frac{|W^\phi_m(1)|^2}{\|\phi\|^2},$
then  \cite[Theorem 1.1]{KL} states that  the 
$w_\phi$-weighted distribution of the eigenvalues $\kappa^\phi_{\mathfrak{p}}$ is asymptotically uniform relative
to the Sato-Tate measure as $\Nr(\mathfrak{N})$ goes to $\infty$. 
That is for any continuous  function $f:[-2,2]\rightarrow \C,$
$$
\lim_{\Nr(\mathfrak{N})\rightarrow \infty \atop {(\mathfrak{p}, \mathfrak{N}) = 1}} \frac{\sum_\phi  f( \kappa_\mathfrak{p}^\phi)   w_\phi }{\sum_\phi w_\phi}=\int_{-2}^2 f(x)\,d\mu_\infty(x).
$$
 A natural question to ask is to get an effective version of the above equidistribution result. For this we define 
the discrete measure $$\tilde{\nu}_{k,\mathfrak{N}}:=\prod_{j=1}^r  \frac{(4\pi)^{k_j-1}}{ (k_j-2)!}
\sum_{{\phi}\in \mathcal{F} }\frac{\delta_{\kappa^\phi_{\mathfrak{p}}} }{\|\phi \|^2}, $$ 
where   $\delta_x$ is the Dirac measure at $x$ and $\mathcal{F}$ be an orthogonal basis for $A_{k}(\mathfrak{N},\omega)$ consisting of eigenfunctions of the Hecke operator $T_{\mathfrak{n}}$.  
Given two probability measures $\mu_1$ and $\mu_2$ on a closed interval $\Omega \subset \mathbb{R},$ the discrepancy between $\mu_1$ and $\mu_2$ is given by 
$$ D(\mu_1,\mu_2) := \sup\{   |\mu_1(I)-\mu_2(I)|\, :\, I=[a,b]\subset \Omega \}.$$
We obtain the following variant of \cite[Theorem 1.6]{JS} and \cite[Theorem 3]{BDS} for the space $A_k(\mathfrak{N},1),$ which gives us information about the rate of convergence of $\tilde{\nu}_{k_l,\mathfrak{N}}$ to $\mu_\infty.$ 
\begin{corollary}\label{application sato tate lower}
  Let the narrow class number of $F$ be equal to 1. Let $d_1,\, \dots \, d_n$ be square-free integers with $F=Q\left( \sqrt{d_1},\, \dots\, , \sqrt{d_n}\right) $ and  $ [\Q\left( \sqrt{d_1},\, \dots\, , \sqrt{d_n}\right):\Q]=2^n=r$. Further let $\mathfrak{b}_1\mathfrak{N}=\tilde{s}\mathcal{O},$ where $\tilde{s}=a\sqrt{\tilde{d}}$, $\tilde{d}=\prod_{l=1}^{t_l} d_l^{\tilde{t}_l}$ with $\tilde{t}_l\in\{0,1\}$ and $a\in \mathbb{Z}$.
 Suppose $S_{1} (m_1,m_2;1;\tilde{s})\neq 0$ for $m_1$ and $m_2$, then
 there exists an infinite sequence of weights $k_l=(k_{l_1},...,k_{l_r})$ with $(k_{l})_0\rightarrow \infty $  such that  $$D(\tilde{\nu}_{k_l,\mathfrak{N}},\mu_\infty)\gg\frac{1}{\big( \log k_{l_j}\big)^2  \times\prod_{i=1}^r (k_{l_i}-1)^{\frac{1}{3}}    }.
 $$
for all $j\in \{1,...\,,r\}.$
\end{corollary}
\begin{proof}
   An analogue of  \cite[Corollary 16]{BDS} can be easily obtained.  We can then apply this analogue version and proceed similarly to the proof of \cite[Theorem 3]{BDS} to prove the corollary.
\end{proof} 
\subsection*{Acknowledgements}
 The author is thankful to  M. Ram Murty and Baskar Balasubramanyam for their valuable suggestions on the article. 
\bibliographystyle{alpha}
\bibliography{Sato-TateNT.bib}

\begin{thebibliography}{CEZA22}

\bibitem[BDS23]{BDS}
Baskar Balasubramanyam, Jishu Das, and Kaneenika Sinha.
\newblock A discrepancy result for {H}ilbert modular forms.
\newblock {\em arXiv:2307.16736v2}, pages 1--23, 2023.

\bibitem[BLS07]{Benjamin}
Elliot Benjamin, Franz Lemmermeyer, and Chip Snyder.
\newblock On the unit group of some multiquadratic number fields.
\newblock {\em Pacific J. Math.}, 230(1):27--40, 2007.

\bibitem[BLT10]{Balu}
R.~Balasubramanian, F.~Luca, and R.~Thangadurai.
\newblock On the exact degree of {$\Bbb Q(\sqrt{a_1,}\sqrt{a_2},\dots,\sqrt{a_\ell})$} over {$\Bbb Q$}.
\newblock {\em Proc. Amer. Math. Soc.}, 138(7):2283--2288, 2010.

\bibitem[CE22]{CEMM}
Mohamed~Mahmoud Chems-Eddin.
\newblock Unit groups of some multiquadratic number fields and 2-class groups.
\newblock {\em Period. Math. Hungar.}, 84(2):235--249, 2022.

\bibitem[CEZA22]{CEMMZ}
Mohamed~Mahmoud Chems-Eddin, Abdelkader Zekhnini, and Abdelmalek Azizi.
\newblock Unit groups of some multiquadratic number fields of degree 16.
\newblock {\em S\~{a}o Paulo J. Math. Sci.}, 16(2):1091--1096, 2022.

\bibitem[Cox89]{Cox}
David~A. Cox.
\newblock {\em Primes of the form {$x^2 + ny^2$}}.
\newblock A Wiley-Interscience Publication. John Wiley \& Sons, Inc., New York, 1989.
\newblock Fermat, class field theory and complex multiplication.

\bibitem[Das24]{JD}
Jishu Das.
\newblock A lower bound for the discrepancy in a {S}ato–{T}ate type measure.
\newblock {\em Ramanujan J.}, 65(2):637 -- 658, 2024.

\bibitem[JS20]{JS}
Junehyuk Jung and Naser~Talebizadeh Sardari.
\newblock Asymptotic trace formula for the {H}ecke operators.
\newblock {\em Math. Ann.}, 378(1-2):513--557, 2020.
\newblock With an appendix by Simon Marshall.

\bibitem[KBM22]{Babu}
C.~G. Karthick~Babu and Anirban Mukhopadhyay.
\newblock Quadratic residue pattern and the {G}alois group of {$\Bbb{Q}(\sqrt{a_1}, \sqrt{a_2}, \dots, \sqrt{a_n})$}.
\newblock {\em Proc. Amer. Math. Soc.}, 150(10):4277--4285, 2022.

\bibitem[KL08]{KL}
Andrew Knightly and Charles Li.
\newblock Petersson's trace formula and the {H}ecke eigenvalues of {H}ilbert modular forms.
\newblock In {\em Modular forms on {S}chiermonnikoog}, pages 145--187. Cambridge Univ. Press, Cambridge, 2008.

\bibitem[Kub56]{Kubota}
Tomio Kubota.
\newblock \"uber den bizyklischen biquadratischen {Z}ahlk\"orper.
\newblock {\em Nagoya Math. J.}, 10:65--85, 1956.

\bibitem[Kur50]{Kuroda}
Sigekatu Kuroda.
\newblock \"uber die {K}lassenzahlen algebraischer {Z}ahlk\"orper.
\newblock {\em Nagoya Math. J.}, 1:1--10, 1950.

\bibitem[LLW14]{Lau-Li-Wang}
Yuk-Kam Lau, Charles Li, and Yingnan Wang.
\newblock Quantitative analysis of the {S}atake parameters of {${\rm GL}_2$} representations with prescribed local representations.
\newblock {\em Acta Arith.}, 164(4):355--380, 2014.

\bibitem[Mor69]{Mordell}
L.~J. Mordell.
\newblock {\em Diophantine equations}.
\newblock Pure and Applied Mathematics, Vol. 30. Academic Press, London-New York, 1969.

\bibitem[MS09]{MSeffective}
M.~Ram Murty and Kaneenika Sinha.
\newblock Effective equidistribution of eigenvalues of {H}ecke operators.
\newblock {\em J. Number Theory}, 129(3):681--714, 2009.

\bibitem[MS10]{MS2}
M.~Ram Murty and Kaneenika Sinha.
\newblock Factoring newparts of {J}acobians of certain modular curves.
\newblock {\em Proc. Amer. Math. Soc.}, 138(10):3481--3494, 2010.

\bibitem[Nic83]{Peter}
Peter~J. Nicholls.
\newblock The minima of indefinite binary quadratic forms.
\newblock {\em J. Number Theory}, 16(1):19--30, 1983.

\bibitem[Pet32]{PH}
Hans Petersson.
\newblock \"{U}ber die {E}ntwicklungskoeffizienten der automorphen {F}ormen.
\newblock {\em Acta Math.}, 58(1):169--215, 1932.

\bibitem[Sel56]{Sel}
A.~Selberg.
\newblock Harmonic analysis and discontinuous groups in weakly symmetric {R}iemannian spaces with applications to {D}irichlet series.
\newblock {\em J. Indian Math. Soc. (N.S.)}, 20:47--87, 1956.

\bibitem[SZ24]{SZ}
Peter Sarnak and Nina Zubrilina.
\newblock Convergence to the {P}lancherel measure of {H}ecke eigenvalues.
\newblock {\em Acta Arith.}, 214:191--213, 2024.

\bibitem[TW16]{TW}
Hengcai Tang and Yingnan Wang.
\newblock Quantitative versions of the joint distributions of {H}ecke eigenvalues.
\newblock {\em J. Number Theory}, 169:295--314, 2016.

\bibitem[Wad66]{Wada}
Hideo Wada.
\newblock On the class number and the unit group of certain algebraic number fields.
\newblock {\em J. Fac. Sci. Univ. Tokyo Sect. I}, 13:201--209 (1966), 1966.

\end{thebibliography}
\end{document}